\documentclass[letterpaper, 12pt]{article}
\usepackage{graphicx}
\usepackage{pgfplots}
\pgfplotsset{compat=1.17}
\usepackage{natbib}
\usepackage{url}
\usepackage{amssymb,amsmath,xcolor,graphicx,xspace,colortbl,rotating} %
\usepackage[raggedrightboxes]{ragged2e} 
\usepackage{subfigure}
\usepackage{appendix}  
\usepackage{bm} 
\usepackage{cancel} 
\usepackage{boxedminipage}  
\usepackage{color}  
\usepackage{endnotes}  
\usepackage{ragged2e}  
\usepackage[onehalfspacing]{setspace}  
\usepackage{tabulary}  
\usepackage{textcomp}  
\usepackage{varioref}  
\usepackage{graphicx}
\graphicspath{ {./images/} }
 
\usepackage[margin=1in]{geometry}
\newtheorem {theorem}{Theorem}

\newtheorem {corollary}{Corollary}

\newtheorem {example}{Example}

\newtheorem {proposition}{Proposition}
\newtheorem {remark}{Remark}

\newenvironment {proof}[1][Proof]{\noindent \textbf {#1.} }{\ \rule {0.5em}{0.5em}}

\makeatletter
\let\@fnsymbol\@arabic
\makeatother

\begin{document}

\title{A Note on the Stability of Monotone Markov Chains }

\author{Bar Light\protect\thanks{Business School and Institute of Operations Research and Analytics, National University of Singapore, Singapore. e-mail: \textsf{barlight@nus.edu.sg} }}  
\maketitle

\thispagestyle{empty}

 \noindent \noindent \textsc{Abstract}:
\begin{quote}
	This note studies monotone Markov chains, a subclass of Markov chains with extensive applications in operations research and economics. While the properties that ensure the global stability of these chains are well studied, their establishment often relies on the fulfillment of a certain splitting condition. We address the challenges of verifying the splitting condition by introducing simple, applicable conditions that ensure global stability. The simplicity of these conditions is demonstrated through various examples including autoregressive processes, portfolio allocation problems and resource allocation dynamics.

\end{quote}



\newpage 

\section{Introduction }

Monotone Markov chains are an important class of Markov chains with widespread applications across fields such as operations research, economics, and the broader realm of applied probability. Their uses span a variety of areas, including queueing systems, savings dynamics, autoregressive processes, and inventory management, to name just a few. The properties of these chains have been extensively studied. This has led to the development of a rich literature that provides various conditions under which these chains exhibit a unique stationary distribution and global stability (see, for example, \cite{dubins1966invariant}, \cite{bhattacharya1988asymptotics}, \cite{hopenhayn1992stochastic}, \cite{lund1996geometric}, \cite{bhattacharya2007random}, \cite{bhattacharya2010limit}, and \cite{kamihigashi2014stochastic}). 

The establishment of these properties often hinges on the fulfillment of a splitting condition (see Section 3.5 in \cite{bhattacharya2007random}), or a weaker condition such as that in Proposition \ref{Prop:MonotoneMark} below. This condition, particularly elegant in the context of infinite state spaces, is often easier to verify than classical irreducibility conditions. In fact, it can sometimes hold in scenarios where irreducibility falls short, as noted in  \cite{bhattacharya1988asymptotics} (see also Example \ref{Examp:AR(1)}). However, despite the theoretical appeal of the splitting condition, verifying its applicability in practical scenarios can be challenging. This note seeks to bridge this gap by introducing a set of easily applicable conditions that imply  global stability. These conditions are satisfied, for instance, when the Markov chain under consideration exhibits specific concavity and contraction properties, which are typically straightforward to verify in applications. 
To demonstrate the applicability of these conditions, we present various examples from existing literature, including AR(1) processes, RCA(1) processes, resource allocation dynamics, and portfolio allocation problems.

\section{Main Results}
 Let $S = I_{1} 
\times \ldots \times I_{n}  \subseteq 
\mathbb{R}^{n}$ where every $I_{i}$ is a nonempty open, half-open, closed, or half-closed interval
in $\mathbb{R}$. We denote by $\mathcal{B}(S)$ the Borel sigma algebra on $S$ and by $\mathcal{P}(S)$ the set of all probability measures on $S$ endowed with the weak topology.

 We consider the Markov chain $\{X_{k}\}_{k=1}^{\infty}$ on $S$ given by
\begin{equation*}
X_{k+1} = w(X_{k},V_{k+1})
    \end{equation*}
where $w: S \times E \rightarrow S $ is a measurable function and $V_{k}$'s are independent and identically distributed random variables on some Polish space $E$, endowed with a partial order $\geq_{E}$ whose principal upper and lower sets are Borel measurable. The shocks are independent of the initial state $X_{1}$. We write $V$ for a random variable with the common distribution of the shocks. 

In this note we focus on monotone Markov chains. We endow $\mathbb{R}^{n}$ with the standard product partial order $\geq$, i.e., $x \geq y$ for $x,y \in \mathbb{R}^{n}$ if $x_{i} \geq y_{i}$ for $i=1,\ldots,n$ and $x > y$ if $x_{i} > y_{i}$ for every $i=1,\ldots,n$. On $E$, we write $v >_{E} u$ if $v \geq_{E} u$ and $v \neq u$. For $x,y \in \mathbb{R}^{n}$ we denote by $\max \{x,y \}$ and $\min \{x,y \}$ the pointwise maximum and pointwise minimum of $x$ and $y$, respectively. A function $f$ from $S \subseteq \mathbb{R}^{n}$ to $\mathbb{R}^{n}$ is called increasing if $x \geq y$ implies $f(x) \geq f(y)$ for all $x,y \in S$ and strictly increasing if $x > y$ implies $f(x) > f(y)$ for all $x,y \in S$. We denote by $[x'',x'] = \{x \in \mathbb{R}^{n} : x'' \leq x \leq x'\}$ an order interval in $\mathbb{R}^{n}$ and by $\mathbb{R}_{+}^{n} =  \{x \in \mathbb{R}^{n} : x \geq 0 \}$ the non-negative part of $\mathbb{R}^{n}$.

For the rest of the note, we assume that $w: S \times E \rightarrow S $ is increasing, i.e., $w(x',v') \geq w(x'',v'')$ whenever $x' \geq x''$ and $v' \geq_{E} v''$.

We denote by $\mu_{k}^{x}$ the law of $X_{k}$ when the initial state is $X_{1}=x$, and by $Q(x,B) = \mathbb{P}(w(x,V) \in B)$ for any $B \in \mathcal{B}(S)$ the Markov transition kernel.

The probability measure $\pi$ on $S$ is called a stationary distribution of $Q$ if $M \pi = \pi$ where $M\mu (B) = \int _{S} Q(x,B) \mu(dx)$ for any $B \in \mathcal{B}(S)$.  
A sequence of probability measures $\{\mu_{k}\}$ on $S$ is called tight if for all $\epsilon > 0$ there is a compact subset $K_{\epsilon}$ of $S$ such that $\mu_{k}(S \setminus K_{\epsilon}) \leq \epsilon$ for all $k$. Tightness is a standard assumption in results on the existence of a stationary distribution (see \cite{meyn2012markov} for an extensive study of stationary distributions).
The Markov kernel $Q$ is called globally stable if $Q$ has a unique stationary distribution $\pi$ and $M^{t}\mu$ converges weakly to $\pi$ for every $\mu \in \mathcal{P}(S)$.

We now present the main theorem of this note. We provide conditions that are typically simple to verify on the function $w$ that imply the global stability of $Q$. We illustrate the usefulness of these conditions after we provide the proof of Theorem \ref{Thm:Main}.  

We will say that $(v',v'')$ is an ordered normal pair if  $v',v'' \in E$, $v' >_{E} v''$, $\mathbb{P}(V \geq_{E} v') > 0$, $\mathbb{P}(V \leq_{E} v'') > 0$, $w(x,v') > w(x,v'')$ for all $x \in S$, and $w(x,v')$, $w(x,v'')$ are continuous in $x$. For a given $v \in E$ a fixed point of $w(\cdot,v)$ is an element $x \in S$ such that $w(x,v)=x$.

\begin{theorem} \label{Thm:Main}
    If the sequence $\{\mu_{k}^{x}\}_{k=1}^{\infty}$ is tight for any $x \in S$ and $(v',v'')$ is an ordered normal pair such that the following two conditions are satisfied:

(i) The function $w(\cdot,v'')$ or the function $w(\cdot,v')$ has a unique fixed point.

(ii) For any $x \in S$, there exist $y'',y' \in S$ such that \begin{equation*}
     y' \geq \max \{w(y',v'),x\} \text{ and } y'' \leq \min \{w(y'',v''),x\}.
     \end{equation*}

Then $Q$ is globally stable. 
\end{theorem}

To prove the theorem, we will need the following two propositions.
Recall that a partially ordered set $(Z , \leq )$ is said to be a lattice if for all $x ,y \in Z$, $\sup \{x ,y\}$ and $\inf \{y ,x\}$ exist in $Z$. The partially ordered set $(Z , \leq )$ is a complete lattice if for all non-empty subsets $Z^{ \prime } \subseteq Z$ the elements $\sup Z^{ \prime }$ and $\inf Z^{ \prime }$ exist in $Z$. For $x'',x'\in S$ with $x''\leq x'$, the order interval $[x'',x']$ is a complete lattice (see Section 2.3 and Theorem 2.3.1 in \cite{topkis2011supermodularity}).

We will use the following Proposition regarding the comparison of fixed points. For a proof, see Corollary 2.5.2 in \cite{topkis2011supermodularity}.

\begin{proposition} \label{Prop: topkis}
\label{Topkis Fixed point}Suppose that $Z$ is a nonempty complete lattice, $E$ is a partially ordered set, and $f (z ,e)$ is an increasing function from $Z \times E$ into $Z$. Then the greatest and least fixed points 
of $f (z ,e)$ exist and are increasing in $e$ on $E$. 
\end{proposition}

We will use Proposition \ref{Prop:MonotoneMark} to provide sufficient conditions that imply the stability of monotone Markov chains. Its order-reversal condition is weaker than the splitting condition discussed in \cite{bhattacharya1988asymptotics}; see also \cite{kamihigashi2014stochastic}. For global stability results with convergence rates that use the splitting condition, see \cite{bhattacharya2010limit}. The next proposition follows from Theorem 6.1 in \cite{kamihigashi2026stationary}, since the state space $S$ satisfies their Assumption 3.1, as noted in their Example 3.1. The independent-chain condition below implies their order-reversing condition.

\begin{proposition} \label{Prop:MonotoneMark}
Assume that for any $x',x'' \in S$ with $x ' \geq x''$ we have $\mathbb{P}(X_{m} ' \leq X_{m}'') > 0 $ for some integer $m$ where  $\{X_{k}'\} _{k=1}^{\infty}$ and $\{X_{k}''\}_{k=1}^{\infty}$ are independent Markov chains and $X_{k}'$ has a law $\mu_{k}^{x'}$ and $X_{k}''$ has a law $\mu_{k}^{x''}$ (so $X_{k}'$ and $X_{k}''$ are monotone Markov chains).  

     Then $Q$ is globally stable if and only if the sequence $\{\mu_{k}^{x}\}_{k=1}^{\infty}$ is tight for any $x \in S$. 
\end{proposition}

We are now ready to prove Theorem \ref{Thm:Main}. We will show that the condition of Proposition \ref{Prop:MonotoneMark} holds.

\begin{proof}
      Let $x' , x'' \in S$ be such that $x' 
      \geq x''$ and suppose that $(v',v'')$ is an ordered normal pair. We assume $w(\cdot,v'')$ has a unique fixed point (the case in which $w(\cdot,v')$ has a unique fixed point follows by reversing the state and shock orders and interchanging $v'$ and $v''$).

      Define the sequences $z_{k +1} = w(z_{k},v')$ and $y_{k +1} = w(y_{k},v'')$ where  $z_{1} = x''$ and $y_{1} =x'$. For the initial states considered first below, we show that $z_{k}$ converges to $C$ and $y_{k}$ converges to $C^{*}$ such that $C> C^{*}$ (see Figure \ref{fig:proof_demo}).

First, suppose that $x'' = z_{1} \leq w(z_{1},v') $ and $x' = y_{1} \geq w(y_{1},v'')$.

The sequence $\{z_{k}\}_{k =1}^{\infty }$ is an increasing  sequence. To see this, first note that $z_{2} = w(z_{1},v')  \geq z_{1}$. If $z_{k} \geq z_{k -1}$ for some $k \geq 2$, then $w$ being increasing  implies that $z_{k +1} =w(z_{k},v') \geq   w(z_{k -1},v') =z_{k}$. Hence, by induction, $ \{z_{k}\}_{k =1}^{\infty }$ is an increasing sequence. Similarly, the sequence $\{y_{k}\}_{k =1}^{\infty }$ is a decreasing sequence.

 From condition (ii), there exist  
$y'$ and $y''$ in $S$ such that $y'' \leq x'' \leq x' \leq y'$ and $w(y',v') \leq y'$ and $y'' \leq w(y'',v'')$. Thus, the restriction $f$ of $w$ to $[y'',y'] \times \{v'',v'\}$ maps into $[y'',y']$ because
$$y'' \leq w(y'',v'') \leq w(y,v'') \leq w(y,v') \leq w(y',v') \leq y'$$ 
for all $y \in [y'',y']$. Note that the inequalities above imply that the sequences $\{z_{k}\}_{k =1}^{\infty }$ and $\{y_{k}\}_{k =1}^{\infty }$ are bounded.

Thus, the bounded and increasing sequence $ \{z_{k}\}_{k =1}^{\infty }$ converges to a limit $C \in [y'',y'] \subseteq S$, which is a fixed point of $w(\cdot,v')$ by the continuity of $w(\cdot,v')$. 
 Similarly, the sequence $\{y_{k}\}_{k =1}^{\infty }$ is decreasing and bounded and therefore converges to a limit $C^{ \ast }$ which is a fixed point of $w(\cdot,v'')$.  Note that $C \leq y'$.

We claim that $C > C^{ \ast }$. The function $f$ is increasing in both arguments. For each selected shock, $f$ is an increasing self-map of the complete lattice $[y'',y']$. From Proposition \ref{Prop: topkis}, the greatest and least fixed points of $f$ are increasing on $\{v'',v'\} $. From condition (i), the function $w(\cdot,v'')$ has a unique fixed point $C^{*}$. Hence, the least fixed point of $w(\cdot,v')$ on $[y'',y']$, say $C'$, satisfies $C^{*} \leq C'$. Using the fact that $w(x,v') > w(x,v'')$ for all $x \in S$, we obtain
$$C'=w(C',v') \geq w(C^{*},v') > w(C^{*},v'')=C^{*}.$$ Because $C \in [y'',y']$ is a fixed point of $w(\cdot,v')$  we have $C' \leq C$ so $C^{*} < C$.

Let $x^{ \ast } =(C +C^{ \ast })/2$ then $x^{*} \in S$. Since $z_{k}\uparrow C$ and $y_{k}\downarrow C^{ \ast }$, there exists a finite $N_{1}$ such that $z_{k} >x^{ \ast }$ for all $k \geq N_{1}$, and similarly, there exists a finite $N_{2}$ such that $y_{k} <x^{ \ast }$ for all $k \geq N_{2}$. Let $m =\max \{2,N_{1},N_{2}\}$. Since $X_{1}$ is the initial state, there are $m-1$ shocks up to time $m$. If all these shocks are at least $v'$, monotonicity implies that the chain starting at $x''$ is at least $z_{m}>x^{*}$ at time $m$. This event has probability
$$[\mathbb{P}(V \geq_{E} v')]^{m-1}>0.$$
Similarly, if all these shocks are at most $v''$, the chain starting at $x'$ is at most $y_{m}<x^{*}$ at time $m$, an event with probability
$$[\mathbb{P}(V \leq_{E} v'')]^{m-1}>0.$$

Thus, we have shown that $\mu_{m}^{x'}(\{x \in S: x < x ^{*} \}) > 0$ and $\mu_{m}^{x''}(\{x \in S: x \geq x ^{*} \}) > 0$ for some $x^{*} \in S$ and some integer $m$ when  $x'' \leq w(x'',v')$ and $x' \geq w(x',v'')$.

Now consider arbitrary $x'',x' \in S$ with $x'' \leq x'$. 
From condition (ii) we can find $y',y'' \in S$ such that 
$$y'' \leq x'' \leq x' \leq y' \text{ and } y'' \leq w(y'',v') \text{, } y' \geq w(y',v'') .$$

Thus, using the argument above, there exists $y^{*}$ such that there is a positive probability to move from the state $y''$ to the set $\{x \in S: x \geq y ^{*} \}$ and to move from the state $y'$ to the set $\{x \in S: x < y^{*} \}$ at time $m$. 

Because $w$ is increasing, $\mu_{m}$ is monotone, i.e., we have $\mu_{m}^{x}(B) \geq \mu_{m}^{y}(B)$ for every upper set $B$ and $x \geq y$ where $B \in \mathcal{B}(S)$ is an upper set if $z \in B$ implies $z' \in B$ for all $z'\in S$ with $z' \geq z$. Thus, 
$$\mu_{m}^{x''}(\{x \in S: x \geq y ^{*} \}) \geq \mu_{m}^{y''}(\{x \in S: x \geq y ^{*} \}) > 0 $$ 
and 
$$\mu_{m}^{x'}(\{x \in S: x < y ^{*} \}) \geq \mu_{m}^{y'}(\{x \in S: x < y ^{*} \}) > 0. $$

We conclude that for any $x'',x' \in S$ such that $x' \geq x''$ we have $\mu_{m}^{x'}(\{x \in S: x < y ^{*} \}) > 0$ and $\mu_{m}^{x''}(\{x \in S: x \geq y ^{*} \}) > 0$ for some $y^{*} \in S$ and some integer $m$. 

Hence, $$\mathbb{P} (X_{m}' \leq X_{m}'') \geq \mathbb{P} (X_{m}' \leq y^{*} ) \mathbb{P} (X_{m}'' \geq y^{*} ) >0 $$
for some $m$ when   $\{X_{k}'\} _{k=1}^{\infty}$ and $\{X_{k}''\}_{k=1}^{\infty}$ are independent Markov chains and $X_{k}'$ has a law $\mu_{k}^{x'}$ and $X_{k}''$ has a law $\mu_{k}^{x''}$.

From Proposition \ref{Prop:MonotoneMark}, $M$ has a unique stationary distribution and $M^{k}\lambda$ converges to the unique stationary distribution for every $\lambda \in \mathcal{P}(S)$. 
\end{proof}

\begin{figure}[h]
    \centering
    \includegraphics[width=0.7\textwidth]{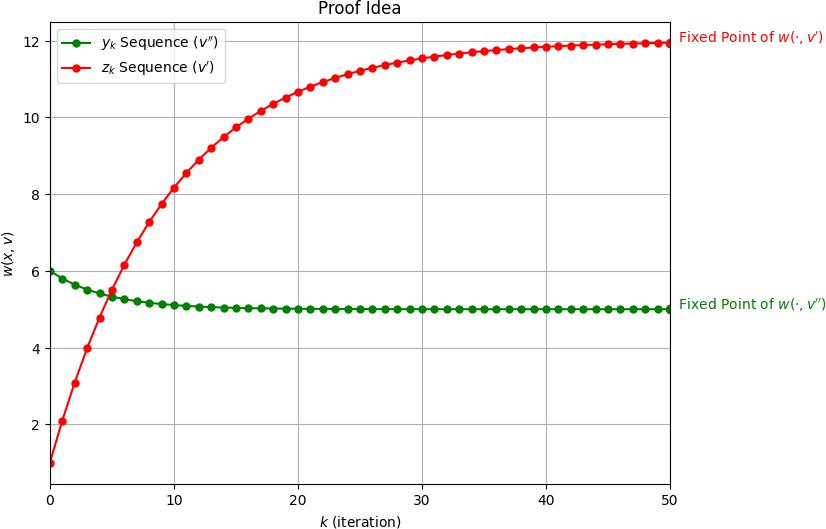}
    \caption{In the proof we construct sequences \(  \{z_{k} \} \), \( \{ y_{k} \} \) such that \( z_{k} \) converges to a fixed point $C$ of \( w(\cdot, v') \) and \( y_{k} \) converges to a fixed point $C^{*}$ of \( w(\cdot, v'') \) and $C > C^{*}$ for suitable initial states $y_{1} \geq z_{1}$. We then show that this together with the fact that $(v',v'')$ is an ordered normal pair implies that the conditions of Proposition \ref{Prop:MonotoneMark} hold.}
    \label{fig:proof_demo}
\end{figure}

If $S$ has a greatest and a least element, then condition (ii) of Theorem \ref{Thm:Main} is satisfied. Indeed, we can choose $y'$ to be the greatest element and $y''$  to be the least element in condition (ii). In this case, it is also immediate that the tightness condition  holds. This leads to the following corollary. 

     \begin{corollary} \label{cor:compact}
            Assume that $S$ has a greatest and a least element.
    
    Suppose that $(v',v'')$ is an ordered normal pair and the function $w(\cdot,v'')$ has a unique fixed point.
Then $Q$ is globally stable. 
     \end{corollary}

In many cases, it is possible to identify a  $v''$ such that $w(\cdot,v'')$ possesses a unique fixed point, 
thereby facilitating the application of Theorem \ref{Thm:Main}. For broader results where a unique fixed point is assured, we may invoke contraction properties (see Corollary \ref{Cor:contraction}) or concavity properties (see Corollary \ref{Cor:concave}) of the function  $w$.

Recall that the maximum norm on $\mathbb{R}^{n}$ is given by $\lVert  x \rVert _{\infty} = \max_{i} |x_{i}|$. A function $f$ from $S$ to $S$ is said to be a contraction with respect to the maximum norm if 
 $\lVert f(x) - f(y) \rVert _{\infty} \leq K  \lVert x - y \rVert_{\infty}$ for some $K \in (0,1)$ and all $x,y \in S$. The next Corollary shows a connection between the contraction properties of $w(\cdot,v)$ to  global stability. 
    
\begin{corollary} \label{Cor:contraction}
   Assume that the sequence $\{\mu_{k}^{x}\}_{k=1}^{\infty}$ is tight for any $x \in S$. 
    
    Suppose that $(v',v'')$ is an ordered normal pair. 
    Assume that $S=\mathbb{R}^{n}$ or $S=\mathbb{R}_{+}^{n}$, and that $w(x,v')$ and $w(x,v'')$ are contractions with respect to the maximum norm.
    Then $Q$ is globally stable. 
\end{corollary}
 
\begin{proof}
 To use Theorem \ref{Thm:Main} we need to show that conditions (i) and (ii) hold. 

 From the Banach fixed-point theorem (see Theorem 3.48 in \cite{aliprantis2006infinite}), condition (i) holds.

 Let $x \in S$ and let $\underline{x} = \min _{i} x_{i}$ and $\overline{x} = \max\{1,x_{1},\ldots,x_{n}\}$.
 Define $y ' = tz \in S$  where $t \geq 1$ is a scalar and $z_{i} = \overline{x}$ for $i=1,\ldots,n$.  Then $y' \geq x$ and $\lVert y' - x \rVert _{\infty} = t\overline{x} - \underline{x}$.  By the contraction property we have 
 $$ w_{i} (y',v') - w_{i}(x,v') \leq K (t\overline{x} - \underline{x})$$
 for $i=1,\ldots,n$ and some $K \in (0,1)$. We have
 $$ K (t\overline{x} - \underline{x}) +  w_{i}(x,v') \leq t\overline{x} \text{   \ if and only if  \ }   \frac{ w_{i}(x,v') - K\underline{x} } {t} \leq (1-K)\overline{x} $$
 which holds for some large $t$. Hence, $w_{i}(y',v') \leq y'_{i}$ for each $i$.

 When $S=\mathbb{R}^{n}$, an analogous argument shows that there exists $y'' \in S$ such that $y'' \leq x$ and $y'' \leq w(y'',v'')$. When $S=\mathbb{R}_{+}^{n}$, choose $y''=0$.
Thus, condition (ii) of Theorem \ref{Thm:Main} holds. 
\end{proof}

A version of Corollary \ref{Cor:contraction} can be adjusted for other norms that are not the maximum norm. Also note that the condition that $w(x,v')$ is a contraction can be relaxed if condition (ii) of Theorem \ref{Thm:Main} holds.

Corollary \ref{Cor:contraction} can be useful in easily proving the global stability of important Markov processes such as autoregressive processes. We now provide a few examples. 

\begin{example}  \label{Examp:AR(1)} (Autoregressive process). 
 Consider the Markov chain on $S=\mathbb{R}^{n}$ which describes an autoregressive process 
\begin{equation*} 
X_{t+1} = w(X_{t},V_{t+1}) = A X_{t} + V_{t+1}
\end{equation*}
    where $V_{t+1}$ is a random vector on $E=\mathbb{R}^{n}$ with $\mathbb{E} | V_{i,t+1} | < \infty$ for $i=1,\ldots,n$ and $A \in \mathbb{R}^{n \times n}_{+}$. Then $w$ is an increasing function when $E$ is endowed with the product order on $\mathbb{R}^{n}$. Suppose that the operator norm $\lVert A \rVert _{\infty}$ is bounded by $0< K < 1$ where $\lVert A \rVert _{\infty} = \sup _{x \neq 0} \lVert Ax \rVert _{\infty} / \lVert x \rVert _{\infty} $. Then $w(\cdot , v)$ is a contraction with respect to the maximum norm for every $v$. If there exist $v'>v''$ with $\mathbb{P}(V\geq v')>0$ and $\mathbb{P}(V\leq v'')>0$, then $(v',v'')$ is an ordered normal pair. The moment bound $\mathbb{E}\lVert X_{t+1}\rVert_{\infty}\leq K\mathbb{E}\lVert X_t\rVert_{\infty}+\mathbb{E}\lVert V\rVert_{\infty}$ implies tightness. Under this shock condition, we can apply Corollary \ref{Cor:contraction} to prove the stability of the Markov chain. The stability conclusion itself holds without this shock condition: the series $\sum_{j=0}^{\infty} A^{j}V_{j+1}$ converges absolutely almost surely, while $A^{t}x\to 0$. Interestingly, one can provide examples of AR(1) processes where standard methods to prove stability such as irreducibility conditions, do not apply (see for example \cite{athreya1986mixing} and \cite{kamihigashi2012order}).
\end{example}

The AR(1) example above can be extended to a positive autoregressive process with random coefficients (RCA) as our next example shows. Note that in this example, $w(\cdot,v)$ is typically not a contraction for every $v$.  

\begin{example}  \label{Examp:RCA(1)} (Nonlinear positive autoregressive process with random coefficients). 
Consider the Markov chain on $S=\mathbb{R}_{+}$ 
\begin{equation} \label{Eq:RCA}
X_{t+1} = w(X_{t},V_{t+1}) := Y_{t+1} f(X_{t}) + Z_{t+1}
\end{equation}
with $V_{t+1}=(Y_{t+1},Z_{t+1})$ where $Z_{t+1}$ is a random variable on $\mathbb{R}_{+}$ with a finite mean, $Y_{t+1}$ is a random variable on  $ \mathbb{R}_{+}$, and $f:\mathbb{R}_{+} \rightarrow \mathbb{R}_{+}$ is an increasing function that is Lipschitz continuous with a constant $1$, i.e., $|f(x_{2}) - f(x_{1}) | \leq |x_{2} - x_{1}|$ for all $x_{1},x_{2} \in \mathbb{R}_{+}$. In particular, $f$ can be the identity function and then the chain describes a positive autoregressive process with a random coefficient  (see \cite{goldie2000stability} and \cite{erhardsson2014conditions}). Then $w$ is an increasing function when $E = \mathbb{R}_{+}^{2}$ is endowed with the product order on $\mathbb{R}^{2}$. We assume that the tightness condition holds (it holds if $\mathbb{E}(Y_{t+1}) < 1$). Suppose $v' = (y',z')$ and $v'' = (y'',z'')$ form an ordered normal pair with $0\leq y''\leq y'<1$. Both selected maps are contractions, so we can apply Corollary \ref{Cor:contraction} to prove the stability of the Markov chain.  The same argument applies in the multidimensional case $S = \mathbb{R}^{n}_{+}$ when the selected maps are contractions in the maximum norm and the ordered normal pair and tightness conditions hold. 
\end{example}

\begin{remark}
      The Markov chain induced by Equation (\ref{Eq:RCA}) plays an important role in economic dynamics. This equation models the wealth evolution of an agent who receives a random income $Z_{t}$ and experiences random gross returns on savings $Y_{t}$. In each period, the agent has a wealth of $X_{t}$ and chooses to save $f(X_{t})$ according to some saving function $f$ that is  typically determined by optimizing the agent's expected discounted utility. Such dynamics of wealth are prevalent in various economic models and have been extensively studied, as surveyed in \cite{benhabib2018skewed}. The Markov chain induced by Equation (\ref{Eq:RCA}) also describes the evolution of firms' states in some dynamic oligopoly models \citep{light2022mean}. 
\end{remark}

We provide an additional example  of a portfolio allocation problem where one can easily prove the global stability of the Markov chain under consideration using Corollary \ref{Cor:contraction}. 

\begin{example} (Portfolio allocation problem). 
Consider the following allocation problem. In each period \( t \), an agent has wealth \( x_{t} \geq 0 \) to invest in two different assets. Asset 1 yields a random return \( R_{1,t+1} \) and asset 2 yields a random return \( R_{2,t+1} \). 
In addition, the random variable \( Z_{t+1} \) describes the withdrawal or addition to the agent's wealth in each period. The agent solves the portfolio optimization problem, and the solution is given by two non-negative functions \( g_{1}(x) \) and \( g_{2}(x) \), such that \( g_{1}(x) + g_{2}(x) = x \). Here, \( g_{i}(x) \) is a continuous function that describes how much the agent invests in asset \( i \) when the agent has a wealth of \( x \). We assume the standard condition that $g_{i}$ is increasing for $i=1,2$, i.e., the agent invests more in each asset when their wealth is higher. The Markov chain on $S=[0,\infty)$ that describes the agent's wealth evolution is given by 
$$X_{t+1} = w(X_{t},V_{t+1}) := \max \{(1+ R_{1,t+1}) g_{1}(X_{t}) + (1+R_{2,t+1}) g_{2}(X_{t}) + Z_{t+1} , 0\}$$
with $V_{t+1}=(R_{1,t+1},R_{2,t+1},Z_{t+1})$. Assume $R_{i,t+1}\geq -1$ almost surely for $i=1,2$, and endow $E=[-1,\infty)^2\times\mathbb{R}$ with the product order. Then $w$ is increasing. Now suppose that $(v',v'')$ is an ordered normal pair such that $v'=(r',r',z')$, $v'' = (r'',r'',z'')$ with $-1<r''<r'<0$, $z'>z''$, $z'>0$. Then 
$$ w(x,v') = \max \{(1+r')g_{1}(x) + (1+r')g_{2}(x) + z',0\} = (1+r')x + z'$$
and $w(\cdot,v')$ is a contraction as $(1+r') 
\in (0,1)$. Similarly, $w(x,v'')=\max\{(1+r'')x+z'',0\}$ is a contraction. Thus, we can use Corollary \ref{Cor:contraction} to prove the stability of the chain, provided the tightness condition holds. For conditions ensuring tightness in related wealth accumulation models, see \cite{ma2020income}.
\end{example}

The next corollary shows that natural concavity properties of $w$ can be leveraged to use Theorem \ref{Thm:Main} and prove the global  stability of monotone Markov chains. 

\begin{corollary} \label{Cor:concave}
   Assume that the sequence $\{\mu_{k}^{x}\}_{k=1}^{\infty}$ is tight for any $x \in S$. 
    
    Suppose that $(v',v'')$ is an ordered normal pair.
    Assume that  $S=\mathbb{R}_{+}^{n}$ and let $w=(w_{1},\ldots,w_{n})$. Suppose that $w(a,v'') > a$, $w(a,v') > a$  for some $a >0$ and $w(b,v'') < b $, $w(b,v') < b $ for some $b > a$. 
    
    Then if $w_{i}(\cdot,v'')$ and $w_{i}(\cdot,v')$  are strictly concave for each $i=1,\ldots,n$, $Q$ is globally stable. 
\end{corollary}

\begin{proof}
 To use Theorem \ref{Thm:Main} we need to show that conditions (i) and (ii) hold. 
    
    By Theorem 3 in \cite{kennan2001uniqueness}, $w(\cdot,v')$ has a unique positive fixed point. Since $(v',v'')$ is an ordered normal pair, $w(0,v')>w(0,v'')\geq 0$. Every fixed point $x$ of $w(\cdot,v')$ therefore satisfies $x=w(x,v')\geq w(0,v')>0$. Thus, this fixed point is unique on $S$, and condition (i) of Theorem \ref{Thm:Main} holds. 

For condition (ii) of Theorem \ref{Thm:Main}, we can choose $y''$ to be the vector $0$ since it is the least element of $S$. In addition, for any $x \in S$,  we can take $\lambda \in (0,1]$ and $b = \lambda y'$ such that $y' \geq x$ and  
$$ b_{i} > w_{i} (b,v') = w_{i}( \lambda y' + (1-\lambda)0,v') \geq \lambda w_{i}(y',v') +(1-\lambda) w_{i}(0,v') \geq \lambda w_{i}(y',v')  $$
so $y_{i}' > w_{i}(y',v')$ for $i=1,\ldots,n$. Thus, $y' \geq \max \{ w(y',v') , x \}$ and condition (ii) of Theorem \ref{Thm:Main} holds. 
\end{proof}

As an example of an application of Corollary \ref{Cor:concave} that does not adhere to contraction properties, consider the following resource allocation dynamics: 

\begin{example} (Resource allocation dynamics). 
Consider an economy with \(k\) firms and \(n\)  resources, such as technological development, research and development, or broadly accessible knowledge. The state space, \(S = \mathbb{R}^n_{+}\), consists of vectors that represent the quantity of each resource \(i\) at period \(t\). Each firm $j$ follows a continuous policy function to produce a non-negative output \(f_{ij}(x_t)\) of resource \(i\) based on the current state \(x_t\). The evolution of each resource \(i\) for the next period, denoted by \(x'_i\), is given by the equation:
\[ x'_{i} = w_{i}(x, V) = \sum _{j=1}^{k}f_{ij}(x) + V_{i}, \]
where \(V_i\) is a non-negative random variable with finite expectations reflecting external influences such as breakthroughs occurring independently of the firms' efforts such as contributions from other sectors or academic research, and \(V = (V_1, \ldots, V_n)\). The function \(f_{ij}(x)\) is assumed to be increasing and strictly concave for each resource \(i\) and firm \(j\), reflecting persistency in output and diminishing returns  in each resource.  Under these assumptions,  $w$ is increasing and $w_{i}(x,v)$ is strictly concave in $x$ for each $v$ and $i=1,\ldots,n$. 

Assume there exist $v'>v''$ in $\mathbb{R}_{+}^{n}$ with $\mathbb{P}(V\geq v')>0$ and $\mathbb{P}(V\leq v'')>0$, so that $(v',v'')$ is an ordered normal pair. Under mild assumptions on $f_{ij}$, the conditions of Corollary \ref{Cor:concave} hold but $w$ is not necessarily a contraction. For instance, it is easy to verify that the conditions of the corollary hold for the case where $f_{ij} (x) = \sum _{l=1}^{n} c_{ijl}x_{l}^{d_{ijl}}$ with $c_{ijl}  \in (0,1)$, $d_{ijl} \in (0,1)$ for all $i,j,l$, and $w$ is not a contraction. In this case, it is easy to see that $w(a,v'') > a$, $w(a,v') > a$  for some $a >0$ and $w(b,v'') < b $, $w(b,v') < b $ for some $b > a$, and that the tightness condition holds. 
\end{example}

While Corollaries \ref{Cor:contraction} and \ref{Cor:concave} provide general sufficient conditions for applying Theorem \ref{Thm:Main} based on contraction and concavity properties, there are instances where Theorem \ref{Thm:Main} can be easily applied directly without these conditions by verifying directly that $w(x,v)$ has a unique fixed point for a certain $v$. As an example, consider a Markov chain defined on $\mathbb{R}$ by the transition function \( w(x,v) =f(x) + v \), where \( v \) is a random variable whose support is \([a, b]\), $a<0, b>0$, and \( f \) is a continuous and  strictly increasing function that is bounded from below. To apply Theorem \ref{Thm:Main}, we can choose $v''=0$ and some $v' \in (0,b)$. Once tightness and condition (ii) hold, the key condition to apply Theorem \ref{Thm:Main} is to check if $w(x,0) = f(x)$ has a unique fixed point. This task is generally straightforward even if \( f \) does not exhibit concavity or contraction properties. For example, consider 
$$ f(x) = (\exp(x) + \delta) 1_{ \{x \leq c \} } + g _{c} (x)  1_{ \{x > c \} }   $$
for some $\delta > -1$ and a parameter $c>0$ where $g_{c}(x)$ is continuous, strictly concave, and strictly increasing with $g_{c}(c) = \exp(c) + \delta$. Suppose that for all $c>0$, there exists $b_{c}$ such that $g_{c}(b_{c}) \leq b_{c}-v'$ for some $b_{c}>c$. Note that $f$ is indeed continuous and strictly increasing. In addition, $f$ is bounded from below. Concavity and $g_c(c)+v'>c$ imply $g_c(y)+v'\leq y$ for all $y\geq b_c$. Hence condition (ii) of Theorem \ref{Thm:Main} holds. Condition (i) holds as $g_c$ is continuous and strictly concave, $g_{c}(c) > c$ and $g_{c}(b_{c}) < b_{c}$ so it has a unique fixed point on $[c, \infty)$ and for $x \leq c$, $\exp(x) + \delta > x$, so $f$ has a unique fixed point. Hence, under the tightness condition of Theorem \ref{Thm:Main}, the Markov chain generated by $w$ is globally stable.

While this example is somewhat unconventional, it illustrates that even when the function \( w \) is neither concave nor a contraction, applying Theorem \ref{Thm:Main} can  be easier than applying other methods for establishing the global stability of the Markov chain generated by \( w \). Once tightness and condition (ii) hold, the key advantage of using Theorem \ref{Thm:Main} lies in its reduction of the global stability property to the verification of a unique fixed point for \( w(x,v) \) for some $v$.

\IfFileExists{ecta.bst}{\bibliographystyle{ecta}}{\bibliographystyle{plainnat}}
\bibliography{mon}

\end{document}